\documentclass [twoside,reqno,  12pt] {amsart}

\usepackage{amsfonts}
\usepackage{amssymb}
\usepackage{a4}
\usepackage{color}

\newtheorem{thm}{Theorem}[section]
\newtheorem{cor}[thm]{Corollary}
\newtheorem{lem}[thm]{Lemma}
\newtheorem{prop}[thm]{Proposition}

\newtheorem{rem}[thm]{Remark}

\theoremstyle{definition}

\numberwithin{equation}{section}

\newcommand{\C}{\mathbb{C}}

\newcommand{\N}{\mathbb{N}}

\newcommand{\R}{\mathbb{R}}

\newcommand{\supp}{\operatorname{supp}}

\def\hat{\widehat}
\def\tilde{\widetilde}
\def \bfo {\begin {eqnarray*} }
\def \efo {\end {eqnarray*} }
\def \ba {\begin {eqnarray*} }
\def \ea {\end {eqnarray*} }
\def \beq {\begin {eqnarray}}
\def \eeq {\end {eqnarray}}
\def \supp {\hbox{supp }}

\def \p {\partial}

\def\hat{\widehat}
\def\tilde{\widetilde}
\def \bfo {\begin {eqnarray*} }
\def \efo {\end {eqnarray*} }
\def \ba {\begin {eqnarray*} }
\def \ea {\end {eqnarray*} }
\def \beq {\begin {eqnarray}}
\def \eeq {\end {eqnarray}}
\def \supp {\hbox{supp }}

\def \p {\partial}

\begin{document}

 \title[Partial data inverse problems for semilinear elliptic equations]{Partial data inverse problems for semilinear elliptic equations with gradient nonlinearities}

\author[Krupchyk]{Katya Krupchyk}

\address
        {K. Krupchyk, Department of Mathematics\\
University of California, Irvine\\
CA 92697-3875, USA }

\email{katya.krupchyk@uci.edu}

\author[Uhlmann]{Gunther Uhlmann}

\address
       {G. Uhlmann, Department of Mathematics\\
       University of Washington\\
       Seattle, WA  98195-4350\\
       USA\\
        and Institute for Advanced Study of the Hong Kong University of Science and Technology}
\email{gunther@math.washington.edu}

\maketitle

\begin{abstract}
We show that the linear span of the set of scalar products of gradients of harmonic functions on a bounded smooth domain $\Omega\subset \R^n$ which vanish on a closed proper subset of the boundary is dense in $L^1(\Omega)$. We apply this density result to solve some partial data inverse boundary problems for a class of semilinear elliptic PDE with quadratic gradient terms.

\end{abstract}

\section{Introduction and statement of results}

Let $\Omega\subset \R^n$, $n\ge 2$, be a connected bounded open set with $C^\infty$ boundary. In the paper \cite{DKSU} it is established that  the linear span of the set of products of harmonic functions in $C^\infty(\overline{\Omega})$, which vanish on a closed proper subset of the boundary, is dense in $L^1(\Omega)$.  This result is motivated by the Calder\'on inverse problem with partial data, see \cite{KS_review} and \cite{Uhlmann_review}  for review,  and it provides the solution of the linearized version of the partial data problem at the zero potential.  The recent works \cite{KU_2019} and \cite{LLLS_new} have exploited this density result to give a solution for the partial data inverse boundary problem for a class of semilinear elliptic PDE.

The purpose of this paper is twofold. First we shall give an extension of the density result of \cite{DKSU} where the set of products of  harmonic functions  which vanish on a closed proper subset of the boundary, is replaced by the set of scalar products of gradients of such functions.  We shall then apply this density result to solve some partial data inverse problems for a class of semilinear elliptic PDE with quadratic gradient terms.

The first result of the paper, extending the corresponding result of \cite{DKSU}, is as follows.
\begin{thm}
\label{thm_product_gradients}
Let $\Omega\subset \R^n$, $n\ge 2$, be a connected bounded open set with $C^\infty$ boundary,  let $ \Gamma\subset \p \Omega$ be an open nonempty subset of $\p \Omega$, and let $\tilde \Gamma=\p \Omega\setminus \Gamma$. Then the  linear span of the set of scalar products of gradients of harmonic functions in $C^\infty(\overline{\Omega})$, which vanish on $\tilde \Gamma$, is dense in $L^1(\Omega)$.
\end{thm}

We shall next proceed to state our results concerning inverse boundary problems for a class of semilinear elliptic PDE with quadratic gradient terms. Specifically, we shall consider the following Dirichlet problem,
\begin{equation}
\label{eq_ref_1}
\begin{cases}
-\Delta u+ q(x)(\nabla u)^2+ V(x,u)=0 \quad \text{in}\quad \Omega, \\
u=f \quad \text{on}\quad \p \Omega.
\end{cases}
\end{equation}
Here $q\in C^\alpha(\overline{\Omega})$  for some $0<\alpha<1$, the H\"older space,  and the function  $V:\overline{\Omega}\times \C\to \C$ satisfies the following conditions:
\begin{itemize}
\item[(i)] the map $\C\ni z\mapsto V(\cdot,z)$ is holomorphic with values in $C^\alpha(\overline{\Omega})$,
\item[(ii)] $V(x,0)=\p_z V(x,0)=\p^2_z V(x,0)=0$, for all $x\in \overline{\Omega}$.
\end{itemize}
We have also written $(\nabla u)^2=\nabla u\cdot \nabla u$.
  It follows from (i) and (ii) that $V$ can be expanded into a power series
\begin{equation}
\label{eq_V}
V(x,z)=\sum_{k=3}^\infty V_k(x) \frac{z^k}{k!},\quad V_k(x):=\p_z^k V(x,0)\in C^\alpha(\overline{\Omega}),
\end{equation}
converging in the $C^\alpha(\overline{\Omega})$ topology.

It is  shown in Appendix \ref{App_well_posedness} that  there exist $\delta>0$ and $C>0$ such that
when  $f\in B_{\delta} (\p \Omega):=\{f\in C^{2,\alpha}(\p \Omega): \|f\|_{C^{2,\alpha}(\p \Omega)}<\delta\}$, the problem \eqref{eq_ref_1} has a unique solution $u=u_f\in C^{2,\alpha}(\overline{\Omega})$ satisfying $\|u\|_{C^{2,\alpha}(\overline{\Omega})}<C \delta$.

Let $\Gamma_1, \Gamma_2\subset \p \Omega$ be  arbitrary non-empty open subsets of the boundary $\p \Omega$. Associated to the problem \eqref{eq_ref_1}, we define the partial Dirichlet--to--Neumann map $\Lambda_{q,V}^{\Gamma_1, \Gamma_2}f = \partial_{\nu} u_f|_{\Gamma_2}$, where $f\in B_{\delta} (\p \Omega)$, $\supp(f)\subset \Gamma_1$. Here  $\nu$ is the unit outer normal to the boundary.

The second  result of this paper is as follows.
\begin{thm}
\label{thm_main_2}
Let $\Omega\subset \R^n$, $n\ge 2$, be a connected bounded open set with $C^\infty$ boundary, and let $\Gamma_1, \Gamma_2\subset \p \Omega$ be arbitrary open non-empty subsets of the boundary $\p \Omega$.  Let $q_1,q_2 \in C^\alpha(\overline{\Omega})$ and  $V^{(1)}, V^{(2)}: \overline{\Omega}\times \C\to \C$ satisfy the assumptions (i) and (ii).  Assume that $\Lambda_{q_1, V^{(1)}}^{\Gamma_1, \Gamma_2} =\Lambda_{q_2,V^{(2)}}^{\Gamma_1, \Gamma_2}$. Then $q_1=q_2$ in $\Omega$ and $V^{(1)}=V^{(2)}$ in $\Omega\times \C$.
\end{thm}

\begin{rem}
To best of our knowledge, Theorem  \ref{thm_main_2} is new even in the full data case $\Gamma_1=\Gamma_2=\p\Omega$.
\end{rem}

\begin{rem}
We would like to emphasize that in Theorem \ref{thm_main_2} the open non-empty sets $\Gamma_1, \Gamma_2\subset \p \Omega$ are completely arbitrary. It may be interesting to note that the corresponding partial data inverse problem is still open in dimensions $n\ge 3$ in the linear setting, even for the linear Schr\"odinger equation $-\Delta u+q(x)u=0$ in $\Omega$, say.  In dimension $n=2$ in the linear setting, the global identifiability in the partial data inverse problem is established in  \cite{IUY_2010} when $\Gamma_1=\Gamma_2$ is an arbitrary open non-empty portion of $\p\Omega$,  and in  \cite{IUY_2011} when $\overline{\Gamma_1} \cap \overline{\Gamma_2}=\emptyset$, provided that some additional geometric assumptions are satisfied. We also refer to \cite{Daude_K_N_2019} for examples of  non-uniqueness in the anisotropic  Calder\'on problem when the Dirichlet and Neumann data are measured on disjoint subsets of the boundary in dimensions $n=2,3$.
\end{rem}

\begin{rem} To motivate the consideration of nonlinear elliptic PDE, discussed in this paper, let us mention that semilinear PDE with quadratic gradient terms occur naturally in the study of harmonic maps, harmonic heat flow maps, as well as Schr\"odinger maps, see \cite{Tataru_conf_2011}, \cite{Bejenaru}.
\end{rem}

Following \cite{LLLS_new}, we shall next discuss inverse boundary problems for semilinear elliptic equations with quadratic gradient terms, in the presence of an unknown obstacle. Let $\Omega\subset \R^n$ be a bounded open set with a connected $C^\infty$  boundary, and let
$D\subset\subset \Omega$ be such that $\Omega\setminus\overline{D}$ is connected and $\p D\in C^\infty$. Let us consider the following boundary  problem,
\begin{equation}
\label{eq_ref_2}
\begin{cases}
-\Delta u+ q(x)(\nabla u)^2+ V(x,u)=0 & \text{in}\quad \Omega\setminus \overline{D}, \\
u=0 &  \text{on}\quad \p D, \\
u=f &  \text{on}\quad \p \Omega.
\end{cases}
\end{equation}
An application of Theorem \ref{thm_well-posedness} of Appendix \ref{App_well_posedness}, as before, gives that for all  $f\in B_{\delta} (\p \Omega)$, the problem \eqref{eq_ref_2} has a unique small solution $u\in C^{2,\alpha}(\overline{\Omega}\setminus D)$.  Let $\Gamma_1, \Gamma_2\subset \p \Omega$ be  arbitrary non-empty open subsets of the boundary $\p \Omega$.  We define the associated partial Dirichlet--to--Neumann map $\Lambda^{D, \Gamma_1, \Gamma_2}_{q,V}$ by
\[
\Lambda^{D,\Gamma_1,\Gamma_2}_{q,V}( f)=\p_\nu u|_{\Gamma_2}, \quad f\in B_{\delta} (\p \Omega), \quad \supp(f)\subset\Gamma_1.
\]
We are interested in the inverse problem of determining the unknown obstacle $D$, the coefficient $q$, and the non-linear term $V$, all from the knowledge of the partial Dirichlet--to--Neumann map $\Lambda^{D,\Gamma_1, \Gamma_2}_{q,V}$.

The following result is analogous to  \cite[Theorem 1.2]{LLLS_new}, with the novelty that we allow quadratic gradient terms in the nonlinearity, and that we can  perform measurements on arbitrary open non-empty sets $\Gamma_1, \Gamma_2\subset \p \Omega$.
\begin{thm}
\label{thm_main_3}
Let $\Omega\subset \R^n$, $n\ge 2$, be a bounded open set with connected $C^\infty$ boundary, and let $\Gamma_1, \Gamma_2\subset \p \Omega$ be arbitrary open non-empty subsets of the boundary $\p \Omega$. Let $D_1, D_2\subset\subset \Omega$ be non-empty open subsets with $C^\infty$ boundaries such that $\Omega\setminus \overline{D_j}$ are connected, $j=1,2$. Let $q_j \in C^\alpha(\overline{\Omega}\setminus D_j)$ and  $V^{(j)}: (\overline{\Omega}\setminus D_j)\times \C\to \C$ satisfy the assumptions (i) and (ii), $j=1,2$.  Assume that $\Lambda^{D_1, \Gamma_1, \Gamma_2}_{q_1,V^{(1)}}=\Lambda^{D_2, \Gamma_1, \Gamma_2}_{q_2,V^{(2)}}$. Then $D:=D_1=D_2$, $q_1=q_2$ in $\Omega\setminus \overline{D} $ and $V^{(1)}=V^{(2)}$ in $(\Omega\setminus \overline{D})\times \C$.
\end{thm}
\begin{rem}
It may be interesting to note that the simultaneous recovery of an obstacle and surrounding potentials in the linear setting, say in the case of the linear Schr\"odinger equation,   constitutes an open problem, see \cite{Isakov_2009},  \cite{LLLS_new} for a discussion.
\end{rem}

Let us remark that inverse boundary problems for nonlinear elliptic PDE have been  studied extensively in the literature. To the best of our knowledge, the following main types of nonlinear scalar equations have been considered, under suitable assumptions on the nonlinearity:
\begin{itemize}

\item[(i)]  $-\Delta u+ a(x,u)=0$, see \cite{IsaSyl_94}, \cite{IsaNach_1995}, \cite{Sun_2010}  for the full data problem in the Euclidean case, and \cite{Feizmohammadi_Oksanen}, \cite{LLLS} for  the manifold case, \cite{IY_2013} for the partial data problem in the $n=2$ case, and \cite{KU_2019},  \cite{LLLS_new}  for the partial data problem when $n \geq 2$,

\item[(ii)] $-\Delta u+b(u,\nabla u)=0$, see \cite{Isakov_2001} for the partial data problem in the case $n=3$,

\item[(iii)] $-\Delta u+q(x,\nabla u)=0$,  see \cite{Sun_2004} for the full data problem when $n=2$,

\item[(iv)] $\nabla \cdot (\gamma(x,u)\nabla u)=0$, see \cite{Sun_96}, \cite{Sun_Uhlm_97} for the full data problem in the case $n\ge 2$,

\item[(v)] $\nabla\cdot (\overrightarrow{C}(x,\nabla u))=0$, see \cite{CNV_2019},  \cite{Kang_Nak_02}, \cite{Hervas_Sun} for the full data problem,

\item[(vi)] $\nabla \cdot (c(u,\nabla u)\nabla u)=0$, see  \cite{Mun_Uhl_2018} for the full data problem when $n\ge 2$.
\end{itemize}

A classical method for attacking inverse boundary problems for nonlinear elliptic PDE, going back to \cite{Isakov_93}, consists of performing a first order linearization of the given nonlinear Dirichlet-to-Neumann map, allowing one to reduce the inverse problem to an inverse boundary problem for a linear elliptic equation, and to employ the available results in this case.  A second order linearization of the nonlinear Dirichlet--to--Neumann map has also been successfully exploited in the works \cite{Assylbekov_Zhou},   \cite{CNV_2019}, \cite{Kang_Nak_02}, \cite{Sun_96}, and \cite{Sun_Uhlm_97}. The recent works \cite{Feizmohammadi_Oksanen}, \cite{LLLS} have introduced a natural and powerful method of higher order linearizations of the nonlinear Dirichlet-to-Neumann map for inverse boundary problems for elliptic PDE, allowing one to solve such problems for nonlinear equations in situations where the corresponding inverse problems in the linear setting are open. This development of inverse boundary problems for nonlinear elliptic PDE was preceded by the pioneering work \cite{KLU18} for inverse problems for nonlinear hyperbolic PDE,  see also \cite{CLOP}, \cite{LUW}, and the references given there.

The problem of determining an unknown obstacle is of central significance in inverse scattering. The first uniqueness result for this problem goes back to Schiffer and  Lax and Phillips \cite[p. 173]{Lax_Phillips_book}.  We refer to the works \cite{Isakov_90}, \cite{Kirsch_Kress_93}, \cite{Kirsch--Paiv_1998} for some other significant contributions,  and to  \cite{Isakov_2009} for a review.

Let us now describe the main ideas of the proofs of Theorem \ref{thm_product_gradients}, Theorem \ref{thm_main_2}, and Theorem \ref{thm_main_3}. First, the proof of Theorem \ref{thm_product_gradients} proceeds similarly to \cite{DKSU}, with the only essential difference being that a certain Runge type approximation theorem needed here has to be established with respect to the $H^1$--topology, as opposed to an $L^2$--approximation result obtained in \cite{DKSU} .

The proof of Theorem \ref{thm_main_2} proceeds by the method of higher order linearizations, with Theorem  \ref{thm_product_gradients} and the main result of \cite{DKSU} being the crucial ingredients.

As for Theorem \ref{thm_main_3}, it is an immediate consequence of Theorem \ref{thm_main_2}, once the obstacle has been recovered. Following \cite{LLLS_new}, the determination of the obstacle is obtained by performing a first order linearization of the problem \eqref{eq_ref_2}, and relying on
a standard contradiction argument. 

The paper is organized as follows. In Section \ref{sec_density} we establish Theorem \ref{thm_product_gradients}. The proof of Theorem \ref{thm_main_2} occupies Section \ref{sec_partial_data}.  Theorem  \ref{thm_main_3}  is proven in Section \ref{sec_proof_obstacle}. In Appendix \ref{App_well_posedness} we show  the well-posedness of the Dirichlet problem for our semilinear elliptic equation with quadratic gradient terms, in the case of small boundary data.

\section{Proof of Theorem \ref{thm_product_gradients}}
\label{sec_density}
We shall follow the strategy of the work \cite{DKSU}. Let $f\in L^\infty(\Omega)$ be such that
\begin{equation}
\label{eq_3_1}
\int_\Omega f \nabla u\cdot \nabla vdx=0,
\end{equation}
for any harmonic functions $u,v\in C^\infty(\overline{\Omega})$  satisfying $u|_{\tilde \Gamma}=v|_{\tilde \Gamma}=0$. In view of the Hahn--Banach theorem, we have to  show that  $f=0$ in $\Omega$. This global statement will be obtained as a corollary of the following local result.

\begin{prop}
\label{prop_local}
Let $\Omega\subset \R^n$, $n\ge 2$, be a bounded open set with $C^\infty$ boundary, let $x_0\in \p \Omega$, and let $\tilde \Gamma\subset \p \Omega$ be the complement of an open boundary neighborhood of $x_0$. Then there exists $\delta>0$ such that if we have \eqref{eq_3_1}  for any harmonic functions $u,v\in C^\infty(\overline{\Omega})$  satisfying $u|_{\tilde \Gamma}=v|_{\tilde \Gamma}=0$, then $f=0$ in $B(x_0, \delta)\cap \Omega$.
\end{prop}

Proposition \ref{prop_local} will be proved in Subsection  \ref{subsection_local}.   The passage from local to global results will be carried out in Subsection \ref{subsection_global}. Here an essential ingredient is a  Runge type approximation theorem in the $H^1$--topology, established in Subsection \ref{subsection_density_gradients} and extending \cite[Lemma 2.2]{DKSU}, where approximation in the $L^2$--sense was shown.

\subsection{Runge type approximation}

\label{subsection_density_gradients} Let $\Omega\subset \R^n$, $n\ge 2$, be a bounded open set with $C^\infty$ boundary,   and let us consider the $L^2$--dual of $H^1(\Omega)$, given by
\[
\tilde H^{-1}(\Omega):=\{v\in H^{-1}(\R^n): \supp(v)\subset \overline{\Omega}\},
\]
see  \cite{Eskin_book}, \cite{McLean_book}.
Here the duality pairing is defined as follows: if  $v\in \tilde H^{-1}(\Omega)$ and $w\in H^1(\Omega)$, then we set
\begin{equation}
\label{eq_2_1}
(v, w)_{\tilde H^{-1}(\Omega), H^1(\Omega)}= (v, \text{Ext}(w))_{H^{-1}(\R^n), H^1(\R^n)},
\end{equation}
where $\text{Ext}(w)\in H^1(\R^n)$ is an arbitrary extension of $w$, and $(\cdot, \cdot)_{H^{-1}(\R^n), H^1(\R^n)}$ is the extension of $L^2$ scalar product $(\varphi, \psi)_{L^2(\R^n)}=\int_{\R^n} \varphi(x)\overline{\psi(x)}dx$.
Note that the definition \eqref{eq_2_1} is independent of the choice of an extension $\text{Ext}(w)$, see \cite[Lemma 22.7]{Eskin_book}.

Let $\Omega_1\subset \Omega_2\subset \R^n$, $n\ge 2$,  be two bounded open sets with smooth boundaries such that
$\Omega_2\setminus\overline{\Omega_1}\ne\emptyset$. Assume that $\p \Omega_1\cap \p \Omega_2=\overline{U}$ where $U\subset \p \Omega_1$ is open with $C^\infty$ boundary.  Associated to $\Omega_2$, we let $\mathcal{G}: C^\infty(\overline{\Omega_2})\to C^\infty(\overline{\Omega_2})$, $a\mapsto w$, be the solution operator to the Dirichlet problem,
\[
\begin{cases} -\Delta w=a & \text{in}\quad \Omega_2,\\
w|_{\p \Omega_2}=0.
\end{cases}
\]
The following result is an extension of \cite[Lemma 2.2]{DKSU}.

\begin{lem}
\label{lem_Runge}
 The space
\[
W:=\{\mathcal{G} a|_{\Omega_1}: a\in C^\infty(\overline{\Omega_2}),\ \emph{\supp}(a)\subset \Omega_2\setminus\overline{\Omega_1}\}
\]
is dense in the space
\[
S:=\{u\in C^\infty(\overline{\Omega_1}): -\Delta u=0 \text{ in }\Omega_1, \ u|_{\p \Omega_1\cap \p \Omega_2}=0\},
\]
with respect to  the $H^1(\Omega_1)$--topology.
\end{lem}
\begin{proof}

We shall use some ideas of  \cite{LLS_poisson}, see also  \cite{Browder}, \cite{LLS_conformal}.  Let $v\in \tilde H^{-1}(\Omega_1)$ be such that
\begin{equation}
\label{eq_1_1}
(v, \mathcal{G}a|_{\Omega_1})_{\tilde H^{-1}(\Omega_1), H^1(\Omega_1)}= 0
\end{equation}
 for any $a\in C^\infty(\overline{\Omega_2})$, $\supp(a)\subset \Omega_2\setminus\overline{\Omega_1}$.   By the Hahn--Banach theorem, it suffices to show that $(v, u)_{\tilde H^{-1}(\Omega_1), H^1(\Omega_1)}=0$ for any $u\in S$.

We have  $\mathcal{G}a\in H^1_0(\Omega_2)$ and let us view $\mathcal{G}a$ as an element of  $H^1(\R^n)$ via an extension by 0 to $\R^n\setminus \Omega_2$. Then there exists a sequence $\varphi_j\in C^\infty_0(\Omega_2)$ such that $\varphi_j\to \mathcal{G}a$ in $H^1(\R^n)$. It follows from \eqref{eq_1_1}  that
\begin{equation}
\label{eq_1_2}
\begin{aligned}
0=(v, \mathcal{G}a)_{H^{-1}(\R^n), H^1(\R^n)}= \lim_{j\to \infty} (v, \varphi_j)_{H^{-1}(\R^n), H^1(\R^n)}\\
=\lim_{j\to \infty} (v, \varphi_j)_{H^{-1}(\Omega_2),H^1_0(\Omega_2)}=
(v, \mathcal{G}a)_{H^{-1}(\Omega_2),H^1_0(\Omega_2)}.
\end{aligned}
\end{equation}

As $v\in \tilde H^{-1}(\Omega_1)$, by \cite[Theorem 3.29]{McLean_book}, there is a sequence $v_j\in C^\infty_0(\Omega_1)$ such that $v_j\to v$ in $H^{-1}(\R^n)$. Let $f\in H^1_0(\Omega_2)$ and $f_j\in C^\infty(\overline{\Omega_2})\cap H^1_0(\Omega_2)$ be the unique solutions to the following Dirichlet problems,
\begin{equation}
\label{eq_1_3}
\begin{cases}
-\Delta f=v & \text{in}\quad \Omega_2,\\
f=0 & \text{on}\quad \p \Omega_2,
\end{cases}\quad \begin{cases}
-\Delta f_j=v_j & \text{in}\quad \Omega_2,\\
f_j=0 & \text{on}\quad \p \Omega_2.
\end{cases}
\end{equation}
Now it follows from \eqref{eq_1_2}, \eqref{eq_1_3} that
\begin{equation}
\label{eq_1_4}
\begin{aligned}
0&=(v, \mathcal{G}a)_{H^{-1}(\Omega_2),H^1_0(\Omega_2)}=\lim_{j\to \infty} (v_j, \mathcal{G}a)_{H^{-1}(\Omega_2),H^1_0(\Omega_2)}\\
 &=\lim_{j\to \infty} (-\Delta f_j, \mathcal{G}a)_{H^{-1}(\Omega_2),H^1_0(\Omega_2)}
 = \lim_{j\to \infty}\int_{\Omega_2} (-\Delta f_j) \overline{\mathcal{G}a} dx\\
 &=\lim_{j\to \infty}\int_{\Omega_2}  f_j \overline{a} dx= \int_{\Omega_2}f\overline{a}dx.
\end{aligned}
\end{equation}
Here in the penultimate equality we use Green's formula,  and the fact that $f_j|_{\p \Omega_2}=\mathcal{G}a|_{\p \Omega_2}=0$. In the last equality in \eqref{eq_1_4} we use that $f_j\to f$ in $L^2(\Omega)$ by the well--posedness of the Dirichlet problem for $-\Delta$ in $\Omega_2$, see \cite[Theorem 23.4]{Eskin_book}. As $a\in C^\infty(\overline{\Omega_2})$, $\supp(a)\subset \Omega_2\setminus\overline{\Omega_1}$, is arbitrary, we get from \eqref{eq_1_4} that $f=0$ in $\Omega_2\setminus \overline{\Omega_1}$.  Since $f\in H^1_0(\Omega_2)$, we see that $f|_{\p \Omega_1\setminus\p \Omega_2}=0$, and therefore, $f|_{\p \Omega_1}=0$. Hence, $f\in H^1_0(\Omega_1)$, and we shall view $f$ as an element of $H^1(\R^n)$ via an extension by $0$ to $\R^n\setminus \Omega_1$.

Let  $\hat f_j\in C^\infty_0(\Omega_1)$ be such that $\hat f_j\to f$ in $H^1(\R^n)$. Thus,  $-\Delta\hat f_j\to -\Delta f $ in $H^{-1}(\R^n)$.  Let $u\in S$ and let $\text{Ext}(u)\in H^1(\R^n)$ be an extension of $u$. Then integrating by parts, we have
\begin{equation}
\label{eq_1_5}
\begin{aligned}
 (-\Delta f, \text{Ext}(u))_{H^{-1}(\R^n), H^1(\R^n)}
&=\lim_{j\to \infty}((-\Delta \hat f_j), \text{Ext}(u))_{H^{-1}(\R^n), H^1(\R^n)}\\
&=\lim_{j\to \infty} \int_{\Omega_1} (-\Delta \hat f_j) \overline{u}dx=0.
\end{aligned}
\end{equation}

We shall consider the compactly supported distribution $g=-\Delta f-v\in H^{-1}(\R^n)$.  As $\supp(v), \supp(f)\subset \overline{\Omega_1}$, we see that $\supp(g)\subset \overline{\Omega_1}$, and it follows from \eqref{eq_1_3} that $\supp(g)\subset \p \Omega_1$.   As $\p \Omega_1$ is a codimension $1$ submanifold in $\R^n$,  by \cite[Theorem 5.1.13]{Agranovich_book},  \cite[Lemma 3.39]{McLean_book}, we obtain that
\[
g=h\otimes\delta_{\p \Omega_1}, \quad h\in H^{-1/2}(\p \Omega_1).
\]
Furthermore, in view of \eqref{eq_1_3}, we have $\supp(g)\subset \p \Omega_1\cap \p \Omega_2=\overline{U}$, and therefore,  $\supp(h)\subset \overline{U}$. Since $U\subset \p \Omega_1$ is an open set with $C^\infty$ boundary,  by \cite[Theorem 3.29]{McLean_book},  there exists a sequence $h_j\in C^\infty_0(U)$ such that $h_j\to h$ in $H^{-1/2}(\p \Omega_1)$.
Hence, we have
\begin{equation}
\label{eq_1_6}
\begin{aligned}
 &(g, \text{Ext}(u))_{H^{-1}(\R^n), H^1(\R^n)}= (h, u|_{\p \Omega_1})_{H^{-1/2}(\p \Omega_1), H^{1/2}(\p \Omega_1)}\\
 &=\lim_{j\to \infty}  (h_j, u|_{\p \Omega_1})_{H^{-1/2}(\p \Omega_1), H^{1/2}(\p \Omega_1)}=\lim_{j\to \infty}  \int_{\p \Omega_1} h_j \overline{u} dS=0,
\end{aligned}
\end{equation}
where in the last equality we use that $u|_{\p \Omega_1\cap \p \Omega_2}=0$.  It follows from \eqref{eq_1_5} and \eqref{eq_1_6} that
\[
(v, u)_{\tilde H^{-1}(\Omega_1), H^1(\Omega_1)} =(-\Delta f, \text{Ext}(u))_{H^{-1}(\R^n), H^1(\R^n)} -(g, \text{Ext}(u))_{H^{-1}(\R^n), H^1(\R^n)} =0.
\]
This completes the proof of Lemma \ref{lem_Runge}.
 \end{proof}

\subsection{From local to global results. Proof of Theorem \ref{thm_product_gradients}}
\label{subsection_global}
We shall follow \cite{DKSU}. We want to show that $f$ vanishes inside $\Omega$. Let us fix a point $x_1\in \Omega$ and let $\theta:[0,1]\to \overline{\Omega}$ be a $C^1$ curve joining $x_0\in \p \Omega\setminus \tilde \Gamma$ to $x_1$ such that $\theta(0)=x_0$, $\theta'(0)$ is the interior normal to $\p \Omega$ at $x_0$ and $\theta(t)\in \Omega$, for all $t\in (0,1]$.  Let us set
\[
\Theta_\varepsilon(t)=\{x\in \overline{\Omega}: d(x, \theta([0,t]))\le \varepsilon\}
\]
so that $\Theta_\varepsilon(t)$ is
a  closed neighborhood of the curve ending at $\theta(t)$, $t\in [0,1]$.  Let
\[
I=\{t\in [0,1]: f\text{ vanishes a.e. on } \Theta_\varepsilon(t)\cap \Omega\}.
\]
Note that by Proposition \ref{prop_local},  $0\in I$ if $\varepsilon>0$ is small enough. One can easily see that $I$ is a closed subset of $[0,1]$. If we show that $I$ is open then as $[0,1]$ is connected,  $I=[0,1]$. Hence, $x_1\notin \supp(f)$. Since
$x_1$ is an arbitrary point of $\Omega$, we have $f=0$ on $\Omega$, completing the proof of Theorem \ref{thm_product_gradients}.

Thus, we only need to show that $I$ is open. To this end, let $t\in I$ and $\varepsilon>0$ be small enough so that $\p \Theta_\varepsilon(t)\cap \p \Omega\subset \p \Omega\setminus\tilde\Gamma$. For $\varepsilon>0$ sufficiently small, the set $\p\Theta_\varepsilon(t)$ intersects $\p \Omega$ transversally, and in suitable local coordinates $y_1,\dots, y_n$ centered at $x_0$, $\p\Omega$ is given by $y_n=0$, and $\p\Theta_\varepsilon(t)$ is given by $y_1=0$.  It is then easy to see that the set $\Omega\setminus \Theta_\varepsilon(t)$ can be smoothed out into an open subset $\Omega_1$ of $\Omega$ with smooth boundary so that
\[
\Omega_1\supset \Omega\setminus \Theta_\varepsilon(t), \quad \p \Omega\cap \p\Omega_1\supset \tilde \Gamma,
\]
and $\p \Omega_1\cap \p \Omega=\overline{U}$ where $U\subset \p \Omega_1$ is an open set with $C^\infty$ boundary.
Furthermore, let us augment the set $\Omega$ by smoothing out the set $\Omega\cup B(x_0, \varepsilon')$, with $0<\varepsilon'\ll \varepsilon$ sufficiently small,  into an open set $\Omega_2$ with smooth boundary so that
\[
\p \Omega_2\cap\p \Omega\supset \p \Omega_1\cap\p \Omega=\p \Omega_1\cap \p\Omega_2 \supset \tilde \Gamma.
\]
Let $G_{\Omega_2}$ be the Green kernel associated to the open set $\Omega_2$,
\[
-\Delta_y G_{\Omega_2}(x,y)=\delta(x-y), \quad G_{\Omega_2}(x,\cdot)|_{\p \Omega_2}=0.
\]
Consider the function
\[
v(x,z)=\int_{\Omega_1} f(y)\nabla_y G_{\Omega_2} (x,y)\cdot \nabla_y G_{\Omega_2} (z,y)dy, \quad x,z\in \Omega_2\setminus\overline{\Omega_1},
\]
which is harmonic in both $x,z\in \Omega_2\setminus\overline{\Omega_1}$. As $f=0$ on $\Theta_\varepsilon(t)\cap \Omega$, we get
\[
v(x,z)=\int_{\Omega} f(y)\nabla_y G_{\Omega_2} (x,y)\cdot \nabla_y G_{\Omega_2} (z,y)dy, \quad x,z\in \Omega_2\setminus\overline{\Omega_1}.
\]
When $x,z\in \Omega_2\setminus\overline{\Omega}$, the Green functions $G_{\Omega_2} (x,\cdot ), G_{\Omega_2} (z,\cdot)\in C^\infty(\overline{\Omega})$ are harmonic on $\Omega$, and $G_{\Omega_2} (x,\cdot )|_{\tilde \Gamma}=G_{\Omega_2} (z,\cdot )|_{\tilde \Gamma}=0$.  By assumption \eqref{eq_3_1}, we have $v(x,z)=0$ when $x,z\in \Omega_2\setminus\overline{\Omega}$. Since $v(x,z)$ is harmonic when $x,z\in \Omega_2\setminus\overline{\Omega_1}$ and $\Omega_2\setminus\overline{\Omega_1}$ is connected, by unique continuation, $v(x,z)=0$ when $x,z\in \Omega_2\setminus\overline{\Omega_1}$, i.e.
\begin{equation}
\label{eq_1_8}
\int_{\Omega_1} f(y)\nabla_y G_{\Omega_2} (x,y)\cdot \nabla_y G_{\Omega_2} (z,y)dy=0, \quad x,z\in \Omega_2\setminus\overline{\Omega_1}.
\end{equation}
Letting $a\in C^\infty(\overline{\Omega_2})$, $\supp(a)\subset \Omega_2\setminus\overline{\Omega_1}$, $b\in C^\infty(\overline{\Omega_2})$, $\supp(b)\subset \Omega_2\setminus\overline{\Omega_1}$, multiplying \eqref{eq_1_8} by $a(x)$, $b(z)$, and integrating, we obtain that
\[
\int_{\Omega_1} f(y)\int_{\Omega_2}\nabla_y G_{\Omega_2} (x,y)a(x)dx\cdot \int_{\Omega_2}\nabla_y G_{\Omega_2} (z,y) b(z)dz dy=0.
\]
Hence, we have
\begin{equation}
\label{eq_1_9}
\int_{\Omega_1} f \nabla u\cdot \nabla v dy=0,
\end{equation}
for all $u,v\in W$.  By continuity of the bilinear form,
\[
H^1(\Omega_1)\times H^1(\Omega_1)\to \C, \quad (u,v)\mapsto \int_{\Omega_1} f\nabla  u\cdot \nabla v dy,
\]
and by Lemma \ref{lem_Runge},  we get \eqref{eq_1_9} for any $u, v\in C^\infty(\overline{\Omega_1})$ harmonic in $\Omega_1$ which vanish on $\p\Omega_1\cap \p\Omega_2$. Now by Proposition \ref{prop_local}, $f$ vanishes on a neighborhood of $\p \Omega\setminus(\p \Omega_1\cap \p \Omega_2)$, and hence, $I$ is an open set.

\subsection{Proof of Proposition \ref{prop_local}}
\label{subsection_local}

We shall follow \cite{DKSU}. First using a conformal transformation of harmonic functions, we reduce to the following setting: $x_0=0$,
the tangent plane to $\Omega$ at $x_0$ is given by $x_1=0$,
\[
\Omega\subset \{ x\in \R^n: |x+e_1|<1\}, \quad \tilde \Gamma =\{x\in \p \Omega: x_1\le -2 c\}
\]
for some $c>0$. Here $e_1=(1,0,\dots, 0)$ be the first coordinate vector.

Let $p(\xi)=\xi^2$, $\xi\in \R^n$, be the principal symbol of $-\Delta$, and let us denote by $p(\zeta)=\zeta^2$ its  holomorphic extension to  $\C^n$. We write
\[
p^{-1}(0)=\{\zeta\in \C^n: \zeta^2=0\}.
\]

Let $\zeta\in p^{-1}(0)$ and let $\chi \in C_0^\infty(\R^n)$ be a cutoff function such that $\chi=1$ on $\tilde \Gamma$. Consider the following function
\begin{equation}
\label{eq_4_0}
u(x,\zeta)=e^{-\frac{i}{h} x\cdot\zeta} +w(x,\zeta),
\end{equation}
where $w$ is the solution to the Dirichlet problem,
\[
\begin{cases} -\Delta w =0 \quad \text{in}\quad \Omega,\\
w|_{\p \Omega}=-(e^{-\frac{i}{h} x\cdot\zeta} \chi)|_{\p \Omega}.
\end{cases}
\]
Thus, $u\in C^\infty(\overline{\Omega})$, $u$ is harmonic in $\Omega$,  and $u|_{\tilde \Gamma}=0$.
We have
\begin{equation}
\label{eq_4_1}
\begin{aligned}
\|w\|_{H^1(\Omega)}\le C\|e^{-\frac{i}{h} x\cdot\zeta} \chi \|_{H^{1/2}(\p \Omega)}\le C \|e^{-\frac{i}{h} x\cdot\zeta} \chi \|_{H^{1}(\p \Omega)}^{1/2} \|e^{-\frac{i}{h} x\cdot\zeta} \chi \|_{L^2(\p \Omega)}^{1/2}\\
\le C(1+h^{-1}|\zeta|)^{1/2} e^{\frac{1}{h}H_K(\text{Im}\, \zeta)},
\end{aligned}
\end{equation}
where $H_K$ is the supporting function of the compact subset $K=\supp\chi\cap \p \Omega$ of the boundary,
\[
H_K(\xi)=\sup_{x\in K} x\cdot \xi, \quad \xi\in \R^n.
\]

Let us take $\chi \in C_0^\infty(\R^n)$ be such that $\supp(\chi)\subset \{x\in \R^n: x_1\le -c\}$ and $\chi=1$ on $\{x\in \p \Omega: x_1\le -2c\}$. Then \eqref{eq_4_1} implies that
\begin{equation}
\label{eq_4_2}
\|w\|_{H^1(\Omega)}\le C (1+h^{-1}|\zeta|)^{1/2}  e^{-\frac{c}{h}\text{Im}\, \zeta_1} e^{\frac{1}{h}|\text{Im}\, \zeta'|},
\end{equation}
when $\text{Im}\,\zeta_1\ge 0$.

The cancellation identity \eqref{eq_3_1} gives that
\begin{equation}
\label{eq_4_3}
\int_{\Omega} f(x) hDu (x,\zeta)\cdot hDu(x,\eta)dx=0,
\end{equation}
for all $\zeta,\eta\in p^{-1}(0)$, where $u(x,\zeta)$, $u(x,\eta)$ are harmonic functions of the form \eqref{eq_4_0}.   It follows from \eqref{eq_4_3} that
\begin{align*}
\int_\Omega f(x) \zeta\cdot\eta &e^{-\frac{ix\cdot(\zeta+\eta)}{h}}dx= \int_\Omega f(x) \zeta e^{-\frac{ix\cdot \zeta}{h}} \cdot hDw (x,\eta)dx\\
&+ \int_\Omega f(x) \eta e^{-\frac{ix\cdot \eta}{h}} \cdot hDw (x,\zeta)dx -\int_{\Omega} f (x)hDw(x,\zeta)\cdot hDw (x,\eta)dx.
\end{align*}
Thus,
\begin{equation}
\label{eq_4_4}
\begin{aligned}
\bigg|\int_\Omega &f(x) \zeta\cdot\eta e^{-\frac{ix\cdot(\zeta+\eta)}{h}}dx\bigg|\le  \|f\|_{L^\infty(\Omega)} \big(|\zeta| \|e^{-\frac{ix\cdot \zeta}{h}}\|_{L^2(\Omega)}\|hDw (x,\eta)\|_{L^2(\Omega)}\\
&+ |\eta| \|e^{-\frac{ix\cdot \eta}{h}}\|_{L^2(\Omega)}\|hDw (x,\zeta)\|_{L^2(\Omega)}+ \|hDw (x,\zeta)\|_{L^2(\Omega)} \|hDw (x,\eta)\|_{L^2(\Omega)}\big).
\end{aligned}
\end{equation}
Now when $\text{Im}\,\zeta_1\ge 0$, using the fact that $\Omega\subset \{x\in \R^n: |x +e_1|<1\}$, we get
\begin{equation}
\label{eq_4_5}
 \|e^{-\frac{ix\cdot \zeta}{h}}\|_{L^2(\Omega)}\le Ce^{\frac{|\text{Im}\, \zeta' |}{h}}, \quad \zeta\in p^{-1}(0).
\end{equation}
We obtain from \eqref{eq_4_4} using \eqref{eq_4_2} and \eqref{eq_4_5} that
\begin{equation}
\label{eq_4_6}
\begin{aligned}
\bigg|\int_\Omega f(x) &\zeta\cdot\eta e^{-\frac{ix\cdot(\zeta+\eta)}{h}}dx\bigg|\le C \|f\|_{L^\infty(\Omega)}  e^{\frac{|\text{Im}\, \zeta' |+ |\text{Im}\, \eta'|}{h}} e^{-\frac{c}{h}\min(\text{Im}\, \zeta_1, \text{Im}\, \eta_1)} \\
&\big( |\zeta|(h^2+h|\eta|)^{1/2}+ |\eta|(h^2+h|\zeta|)^{1/2} + (h^2+h|\zeta|)^{1/2} (h^2+h|\eta|)^{1/2}\big),
\end{aligned}
\end{equation}
for all $\zeta,\eta\in p^{-1}(0)$, $\text{Im}\,\zeta_1\ge 0$, $\text{Im}\, \eta_1\ge 0$.

As in  \cite{DKSU}, consider the map
\[
s:p^{-1}(0)\times p^{-1}(0)\to \C^n, \quad (\zeta,\eta)\mapsto \zeta+\eta.
\]
Its differential at a point $(\zeta_0,\eta_0)$,
\[
Ds(\zeta_0,\eta_0) : T_{\zeta_0}p^{-1}(0)\times T_{\eta_0} p^{-1}(0)\to \C^n, \quad (\zeta,\eta)\mapsto \zeta+\eta,
\]
is surjective provided that $\C^n=T_{\zeta_0}p^{-1}(0)+T_{\eta_0} p^{-1}(0)$, i.e. $\zeta_0$ and $\eta_0$ are linearly independent.  In particular, the latter is true if $\zeta_0=\gamma$ and $\eta_0=-\overline{\gamma}$ with $\gamma=(i,1, 0, \dots, 0)\in \C^n$.  Now $\zeta_0+\eta_0=2i e_1$, and therefore, the inverse function theorem implies that there exists $\varepsilon>0$ small such that any $z\in \C^n$, $|z-2ie_1|<2\varepsilon$, may be decomposed as $z=\zeta+\eta$ where $\zeta, \eta\in p^{-1}(0)$, $|\zeta-\gamma|<C_1\varepsilon$ and $|\eta+\overline{\gamma}|<C_1\varepsilon$ with some $C_1>0$. Furthermore, by rescaling, any $z\in \C^n$ such that  $|z-2i ae_1|<2\varepsilon a$ for some $a>0$, may be decomposed as
\begin{equation}
\label{eq_4_7}
z=\zeta+\eta, \quad \zeta, \eta\in p^{-1}(0),\quad  |\zeta-a \gamma|<C_1a \varepsilon, \quad |\eta+a\overline{\gamma}|<C_1a\varepsilon.
\end{equation}
Now \eqref{eq_4_7} implies that $|\text{Im}\,\zeta'|<C_1a\varepsilon$, $|\text{Im}\,\eta'|<C_1a\varepsilon$, $|\zeta|\le Ca$, and $|\eta|\le Ca$.
If $\varepsilon>0$ is small enough, \eqref{eq_4_7} gives that $\text{Im}\, \zeta_1>a/2$, $\text{Im}\, \eta_1>a/2$, and
$|\zeta\cdot\eta|\ge a^2$,
Hence, it follows from \eqref{eq_4_6} and \eqref{eq_4_7} that
\begin{equation}
\label{eq_4_8}
\bigg|\int_\Omega f(x) e^{-\frac{ix\cdot z}{h}}dx\bigg|\le C \|f\|_{L^\infty(\Omega)} e^{-\frac{ca}{2h}}e^{\frac{2C_1 a\varepsilon}{h}}[a^{-1}(1+a)^{1/2}+a^{-2}(1+a)],
\end{equation}
for all $z\in \C^n$ such that $|z-2i ae_1|<2\varepsilon a$ for some $a>0$ and  $\varepsilon>0$ sufficiently small.
Following  \cite{DKSU} and choosing $a>1$ large, we see that the bound
\eqref{eq_4_8} is completely analogous to the estimate (3.8) in \cite{DKSU}. We may therefore complete the proof of Proposition \ref{prop_local} by repeating the arguments of  \cite{DKSU} exactly as they stand.

\section{Proof of Theorem \ref{thm_main_2}}

\label{sec_partial_data}

We shall first establish that the knowledge of the partial Dirichlet--to--Neumann map $\Lambda^{\Gamma_1, \Gamma_2}_{q,V}$ allows us to recover the coefficient $q$ in the quadratic gradient term in \eqref{eq_ref_1}. To that end, let $\varepsilon=(\varepsilon_1,  \varepsilon_2)\in \C^2$, and let $f_k\in C^\infty(\p \Omega)$, $\supp(f_k)\subset \Gamma_1$, $k=1,2$. An application of Theorem \ref{thm_well-posedness}  shows that for all $|\varepsilon|$ sufficiently small, the Dirichlet problem
\begin{equation}
\label{eq_10_2}
\begin{cases}
-\Delta u_j+ q_j(x)(\nabla u_j)^2+ \sum_{k=3}^\infty V_k^{(j)}(x)\frac{u_j^k}{k!}=0 \quad \text{in}\quad \Omega, \\
u_j=\varepsilon_1 f_1 +\varepsilon_2 f_2 \quad \text{on}\quad \p \Omega,
\end{cases}
\end{equation}
$j=1,2$, has a unique small solution $u_j=u_j(\cdot,\varepsilon)\in C^{2,\alpha}(\overline{\Omega})$, which depends holomorphically on $\varepsilon\in \text{neigh}(0,\C^2)$ with values in $C^{2,\alpha}(\overline{\Omega})$. We shall now carry out a second order linearization of the problem \eqref{eq_10_2} and of the corresponding partial Dirichlet--to--Neumann maps.  Differentiating \eqref{eq_10_2} with respect to $\varepsilon_l$, $l=1,2$, taking $\varepsilon=0$,  and using that $u_j(x,0)=0$, we get
\begin{equation}
\label{eq_10_3}
\begin{cases}
\Delta v_j^{(l)}=0 \quad \text{in}\quad \Omega, \\
v_j^{(l)} = f_l \quad \text{on}\quad \p \Omega,
\end{cases}
\end{equation}
where $v_j^{(l)}= \p_{ \varepsilon_l} u_j|_{\varepsilon=0}$. By the uniqueness and the elliptic regularity for the Dirichlet problem \eqref{eq_10_3}, we see that  $v^{(l)}:=v_1^{(l)}=v_2^{(l)}\in C^\infty(\overline{\Omega})$, $l=1,2$.

Applying $\p_{\varepsilon_1}\p_{\varepsilon_2}|_{\varepsilon=0}$ to \eqref{eq_10_2}, we get
\begin{equation}
\label{eq_10_4}
\begin{cases}
-\Delta (\p_{\varepsilon_1}\p_{\varepsilon_2} u_j |_{\varepsilon=0}) + 2q_j(x) \nabla \p_{\varepsilon_1}  u_j|_{\varepsilon=0} \cdot \nabla \p_{\varepsilon_2} u_j|_{\varepsilon=0} =0 \quad \text{in}\quad \Omega, \\
\p_{\varepsilon_1}\p_{\varepsilon_2} u_j |_{\varepsilon=0} = 0 \quad \text{on}\quad \p \Omega,
\end{cases}
\end{equation}
and letting $w_j=  \p_{\varepsilon_1} \p_{\varepsilon_2} u_j|_{\varepsilon=0}$,  \eqref{eq_10_4} yields that
\begin{equation}
\label{eq_10_5}
\begin{cases}
-\Delta w_j+ 2 q_j(x) \nabla v^{(1)} \cdot \nabla v^{(2)}=0 \quad \text{in}\quad \Omega, \\
w_j = 0 \quad \text{on}\quad \p \Omega.
\end{cases}
\end{equation}
The fact that $\Lambda_{q_1, V^{(1)}}^{\Gamma_1, \Gamma_2}(\varepsilon_1 f_1+\varepsilon_2 f_2)=\Lambda_{q_2, V^{(2)}}^{\Gamma_1, \Gamma_2}(\varepsilon_1 f_1+\varepsilon_2 f_2)$ for all small $\varepsilon_1, \varepsilon_2$ and all $f_1, f_2\in C^\infty(\p \Omega)$ with $\supp(f_1), \supp(f_2)\subset \Gamma_1$ implies that $\p_\nu u_1|_{\Gamma_2}=\p_\nu u_2|_{\Gamma_2}$.  Hence, an application of $\p_{\varepsilon_1}\p_{\varepsilon_2}|_{\varepsilon=0}$ gives $\p_\nu w_1|_{\Gamma_2}=\p_\nu w_2|_{\Gamma_2}$. Multiplying \eqref{eq_10_5} by $v^{(3)} \in C^{\infty}(\overline{\Omega})$ harmonic in $\Omega$ and applying Green's formula,  we get
\[
2\int_{\Omega}(q_1-q_2)(\nabla v^{(1)}\cdot \nabla v^{(2)}) v^{(3)}dx=\int_{\p \Omega\setminus \Gamma_2} (\p_\nu w_1- \p_\nu w_2)v^{(3)}dS=0,
\]
provided that $\supp(v^{(3)}|_{\partial \Omega})\subset \Gamma_2$. Hence, we obtain that
\[
\int_{\Omega}(q_1-q_2)(\nabla v^{(1)}\cdot \nabla v^{(2)}) v^{(3)}dx=0
\]
for any $v^{(l)}\in C^{\infty}(\overline{\Omega})$ harmonic in $\Omega$, $l=1,2,3$,  such that $\supp(v^{(l)}|_{\partial \Omega})\subset \Gamma_1$, $l=1,2$, and $\supp(v^{(3)}|_{\partial \Omega})\subset \Gamma_2$. Taking  $v^{(3)}\not\equiv 0$ and applying Theorem \ref{thm_product_gradients}, we obtain that
\[
(q_1-q_2)v^{(3)}=0 \quad \text{in}\quad \Omega.
\]
Now $v^{(3)}$ is harmonic and therefore,  the set $(v^{(3)})^{-1}(0)$ is of measure zero, see \cite{Mityagin}. Hence $q_1=q_2=:q$ in $\Omega$.

We now come to prove that $V^{(1)}=V^{(2)}$.  To that end, it suffices to show that $V_m^{(1)}=V_m^{(2)}$ for all $m\ge 3$, see \eqref{eq_10_2}, which will be done inductively. Let $\varepsilon=(\varepsilon_1,\dots, \varepsilon_m)\in \C^m$, $m\ge 3$, be small, and $f_k\in C^\infty(\p \Omega)$, $\supp(f_k)\subset \Gamma_1$, $k=1,\dots, m$.  Let $u_j=u_j(\cdot, \varepsilon)\in C^{2,\alpha}(\overline{\Omega})$, $j=1,2$,  be the unique small solution to the Dirichlet problem
\begin{equation}
\label{eq_10_2_m}
\begin{cases}
-\Delta u_j+ q(x)(\nabla u_j)^2+ \sum_{k=3}^\infty V_k^{(j)}(x)\frac{u_j^k}{k!}=0 \quad \text{in}\quad \Omega, \\
u_j=\varepsilon_1 f_1+\dots +\varepsilon_m f_m \quad \text{on}\quad \p \Omega.
\end{cases}
\end{equation}
We shall  first discuss the case $m=3$.  The first linearization of \eqref{eq_10_2_m} leads to the problem \eqref{eq_10_3} with $l=1,2,3$, and therefore, $\p_{\varepsilon_l}u_1|_{\varepsilon=0}=\p_{\varepsilon_l}u_2|_{\varepsilon=0}=:v^{(l)}$, $l=1,2,3$. The secord linearization of \eqref{eq_10_2_m} gives rise  to a problem of the form \eqref{eq_10_4} with $q_j=q$, and therefore, $\p_{\varepsilon_{l_1}}\p_{\varepsilon_{l_2}}u_1|_{\varepsilon=0}=\p_{\varepsilon_{l_1}}\p_{\varepsilon_{l_2}}u_2|_{\varepsilon=0}=:w^{(l_1, l_2)}$, $l_1, l_2\in \{1,2,3\}$. Applying $\p_{\varepsilon_1}\p_{\varepsilon_2}\p_{\varepsilon_3}|_{\varepsilon=0}$ to \eqref{eq_10_2_m}, we obtain that
\[
\begin{cases}
-\Delta w_j+ V_3^{(j)} v^{(1)}v^{(2)}v^{(3)} =H  \quad \text{in}\quad \Omega, \\
w_j = 0 \quad \text{on}\quad \p \Omega,
\end{cases}
\]
where $H(x)=- 2 q(x) [\nabla v^{(1)} \cdot \nabla w^{(2,3)}+ \nabla v^{(2)} \cdot \nabla w^{(1,3)}+ \nabla v^{(3)} \cdot \nabla w^{(1,2)}]$ is independent of $j$. It follows that
\[
\int_{\Omega} (V_3^{(1)}-V_3^{(2)})v^{(1)}v^{(2)}v^{(3)}v^{4}dx=0,
\]
for any $v^{(l)}\in C^{\infty}(\overline{\Omega})$ harmonic in $\Omega$, $l=1,\dots, 4$, such that $\supp(v^{(l)}|_{\partial \Omega})\subset \Gamma_1$, $l=1,2, 3$, and $\supp(v^{(4)}|_{\partial \Omega})\subset \Gamma_2$. Arguing as in \cite{KU_2019}, using the density result of \cite{DKSU}, we conclude that $V_3^{(1)}=V_3^{(2)}$.  The general inductive argument can now be carried out exactly as in  \cite{KU_2019}. The proof of Theorem \ref{thm_main_2} is complete.

\section{Proof of Theorem \ref{thm_main_3}}

\label{sec_proof_obstacle}

Theorem \ref{thm_main_3} is an immediate consequence of Theorem \ref{thm_main_2}, once the obstacle has been recovered. The proof of the fact that $D_1=D_2$ is standard, see for instance \cite{LLLS_new}, and is presented here for completeness and convenience of the reader.

Following \cite{LLLS_new}, we proceed by performing a first order linearization of the problem \eqref{eq_ref_2}. To that end, let $\varepsilon\in \C$, and let $f\in C^\infty(\p \Omega)$, $\supp(f)\subset \Gamma_1$.  An application of Theorem \ref{thm_well-posedness}  shows that for all $|\varepsilon|$ sufficiently small, the Dirichlet problem
\begin{equation}
\label{eq_20_1}
\begin{cases}
-\Delta u_j+ q_j(x)(\nabla u_j)^2+ V_j(x,u_j)=0 & \text{in}\quad \Omega\setminus \overline{D_j}, \\
u_j=0 &  \text{on}\quad \p D_j, \\
u_j=\varepsilon f &  \text{on}\quad \p \Omega,
\end{cases}
\end{equation}
$j=1,2$, has a unique small solution $u_j=u_j(\cdot,\varepsilon)\in C^{2,\alpha}(\overline{\Omega}\setminus D_j)$, which depends holomorphically on $\varepsilon\in \text{neigh}(0,\C)$ with values in $C^{2,\alpha}(\overline{\Omega}\setminus D_j)$. Differentiating \eqref{eq_20_1} with respect to $\varepsilon$,  taking $\varepsilon=0$, and writing $v_j=\p_\varepsilon u_j|_{\varepsilon=0}$, we get
\begin{equation}
\label{eq_20_2}
\begin{cases}
-\Delta v_j=0 & \text{in}\quad \Omega\setminus \overline{D_j}, \\
v_j=0 &  \text{on}\quad \p D_j, \\
v_j= f &  \text{on}\quad \p \Omega.
\end{cases}
\end{equation}
$j=1,2$. The fact that $\Lambda^{D_1, \Gamma_1, \Gamma_2}_{q_1, V^{(1)}}(\varepsilon f)=\Lambda^{D_2, \Gamma_1, \Gamma_2}_{q_2, V^{(2)}}(\varepsilon f)$ for all small $\varepsilon$ and all $f\in C^\infty(\p \Omega)$ with $\supp(f)\subset \Gamma_1$ implies that  $\p_\nu v_1|_{\Gamma_2}=\p_\nu v_2|_{\Gamma_2}$.

Assume that $D_1\ne D_2$, and assume for example that $D_2$ is not contained in $D_1$. Let $G$ be the connected component of $\Omega\setminus({\overline{D_1}}\cup \overline{D_2})$ whose boundary contains $\p \Omega$. Then there exists a point $x_0\in \p D_2$ such that $x_0\in \Omega\setminus \overline{D_1}$ and $x_0\in \p G$, see
 \cite[p. 1579]{Isakov_90}. We reproduce the argument of \cite{Isakov_90} for completeness and convenience of the reader. Indeed, by our assumption and the fact that $\p D_1$ is smooth, there is a point $x'\in D_2\setminus\overline{D_1}$. Let $x''\in G$ be arbitrary and since $\Omega\setminus\overline{D_1}$ is connected, there is a continuous path $s(t)\in \Omega\setminus\overline{D_1}$, for $t\in [0,1]$, such that $s(0)=x'$ and $s(1)=x''$. We let $x_0=s(t_0)$ where $t_0=\sup\{t: s(t)\in D_2\}$.

To complete the proof, we follow \cite{LLLS_new} and let $v=v_1-v_2$. Then we have $-\Delta v=0$ in $G$, $v|_{\p \Omega}=0$, and $\p_\nu v|_{\Gamma_2}=0$. By the unique continuation principle for harmonic functions and continuity of harmonic functions up to the boundary, we conclude that $v_1=v_2$ in $\overline{G}$.  In view of \eqref{eq_20_2}, we get $0=v_2(x_0)=v_1(x_0)$. Let us fix some $f\in C^\infty(\p \Omega)$, $\supp(f)\subset \Gamma_1$, such that $f\ge 0$, $f\not\equiv 0$. As $x_0\in \Omega\setminus\overline{D_1}$, the maximum principle yields that $v_1\equiv 0$ in $\Omega\setminus\overline{D_1}$. Since $v_1$ is continuous up to the boundary of  $\Omega\setminus\overline{D_1}$, we get a contradiction, and therefore, $D_1=D_2$. The proof of Theorem \ref{thm_main_3} is complete.

\begin{appendix}

\section{Well-posedness of the Dirichlet problem for a class of  semilinear elliptic equations with a quadratic gradient term}
\label{App_well_posedness}

The purpose of this appendix is to show the well-posedness of the Dirichlet problem for a class of  semilinear elliptic equations with small boundary data. The argument is standard and is given here for completeness and convenience of the reader.

Let $\Omega\subset \R^n$, $n\ge 2$, be a bounded open set with $C^\infty$ boundary. Let $k\in \N\cup \{0\}$ and $0<\alpha<1$.  The H\"older space  $C^{k,\alpha}(\overline{\Omega})$ consists of all functions $u\in C^k(\overline{\Omega})$ such
\[
\|u\|_{C^{k,\alpha}(\overline{\Omega})}:=\sum_{|\alpha|=k} \sup_{x,y\in \Omega, x\ne y}\frac{|\p^\alpha u(x)-\p^\alpha u(y)|}{|x-y|^\alpha}+\|u\|_{L^\infty(\Omega)}<\infty.
\]
We shall write $C^{\alpha}(\overline{\Omega})=C^{0,\alpha}(\overline{\Omega})$. For future reference,  we remark that $C^{k,\alpha}(\overline{\Omega})$ is an  algebra under pointwise multiplication, and
\begin{equation}
\label{eq_app_ref_01}
\|uv\|_{C^{k, \alpha}(\overline{\Omega})} \le C\big(\|u\|_{C^{k,\alpha}(\overline{\Omega})} \|v\|_{L^\infty(\Omega)}+\|u\|_{L^\infty(\Omega)} \|v\|_{C^{k,\alpha}(\overline{\Omega})} \big), \quad u,v\in C^{k,\alpha}(\overline{\Omega}),
\end{equation}
see \cite[Theorem A.7]{Hormander_1976}. We also have the corresponding spaces $C^{k,\alpha}(M)$, where $M$ is a compact $C^\infty$ manifold.

We shall be concerned with the following Dirichlet problem,
\begin{equation}
\label{eq_app_ref_1}
\begin{cases}
-\Delta u+ q(x)(\nabla u)^2+ V(x,u)=0 \quad \text{in}\quad \Omega, \\
u=f \quad \text{on}\quad \p \Omega.
\end{cases}
\end{equation}
Here $q\in C^\alpha(\overline{\Omega})$, for some $0<\alpha<1$, and the function  $V:\overline{\Omega}\times \C\to \C$ satisfies the following conditions:
\begin{itemize}
\item[(a)] the map $\C\ni z\mapsto V(\cdot,z)$ is holomorphic with values in the H\"older space  $C^\alpha(\overline{\Omega})$,
\item[(b)] $V(x,0)=0$, for all $x\in \overline{\Omega}$.
\end{itemize}
The condition (b) ensures that $u=0$ is a solution to \eqref{eq_app_ref_1}  when $f=0$.  It follows from (a) and (b) that $V$ can be expanded into a power series
\begin{equation}
\label{eq_app_V}
V(x,z)=\sum_{k=1}^\infty V_k(x) \frac{z^k}{k!},\quad V_k(x):=\p_z^k V(x,0)\in C^\alpha(\overline{\Omega}),
\end{equation}
converging in the $C^\alpha(\overline{\Omega})$ topology.  Assume for simplicity that $V_1\in C^\infty(\overline{\Omega})$ and  let us suppose furthermore that
\begin{itemize}
\item[(c)]
$0$ is not a Dirichlet eigenvalue of $-\Delta+V_1$.
\end{itemize}

We have the following result.

\begin{thm}
\label{thm_well-posedness}
There exist $\delta>0$, $C>0$ such that for any $f\in  B_{\delta}(\p \Omega):=\{f\in C^{2,\alpha}(\p \Omega): \|f\|_{C^{2,\alpha}(\p \Omega)}< \delta\}$, the problem \eqref{eq_app_ref_1} has a solution $u=u_f\in C^{2,\alpha}(\overline{\Omega})$ which satisfies
\[
\|u\|_{C^{2,\alpha}(\overline{\Omega})}\le C\|f\|_{C^{2,\alpha}(\p \Omega)}.
\]
Furthermore, the solution $u$ is unique within the class $\{u\in C^{2,\alpha}(\overline{\Omega}): \|u\|_{C^{2,\alpha}(\overline{\Omega})}< C\delta \}$ and
 it is depends holomorphically on $f\in B_\delta(\p \Omega)$.
\end{thm}

\begin{proof}
We shall follow \cite{LLLS} and in order to prove this result we shall rely on the implicit function theorem for holomorphic maps between complex Banach spaces, see \cite[p. 144]{Poschel_Trub_book}. To that end,  let us set
\[
B_1=C^{2,\alpha}(\p \Omega), \quad  B_2=C^{2,\alpha}(\overline{\Omega}), \quad B_3=C^{\alpha}(\overline{\Omega})\times C^{2,\alpha}(\p \Omega),
\]
and consider the map,
\begin{equation}
\label{eq_app_ref_2}
F:B_1\times B_2\to B_3, \quad F(f,u)=(-\Delta u+ q(x)(\nabla u)^2+ V(x,u), u|_{\p \Omega}-f).
\end{equation}
Let us first show that the map $F$ has indeed the mapping property given in \eqref{eq_app_ref_2}. We have $-\Delta u\in C^{\alpha}(\overline{\Omega})$ and an application of \eqref{eq_app_ref_01} gives $q(x)(\nabla u)^2\in C^{\alpha}(\overline{\Omega})$. We only need to check that $V(x,u(x))\in C^{\alpha}(\overline{\Omega})$.  To this end, let us first observe that by Cauchy's estimates, the coefficients $V_k(x)$ in  \eqref{eq_app_V} satisfy
\begin{equation}
\label{eq_app_ref_3}
\|V_k\|_{C^{\alpha}(\overline{\Omega})}\le \frac{k!}{R^k}\sup_{|z|=R}\|V(\cdot, z)\|_{C^{\alpha}(\overline{\Omega})}, \quad R>0.
\end{equation}
Using \eqref{eq_app_ref_01} and \eqref{eq_app_ref_3}, we get for all $k=1,2,\dots$,
\begin{equation}
\label{eq_app_ref_4}
\bigg\| \frac{V_k}{k!} u^k \bigg\|_{C^{\alpha}(\overline{\Omega})}\le \frac{C^k}{R^k}\|u\|_{C^{\alpha}(\overline{\Omega})}^k \sup_{|z|=R}\|V(\cdot, z)\|_{C^{\alpha}(\overline{\Omega})}.
\end{equation}
Choosing $R=2C\|u\|_{C^{\alpha}(\overline{\Omega})}$, we see that the series $\sum_{k=1}^\infty V_k(x) \frac{z^k}{k!}$
converges in $C^{\alpha}(\overline{\Omega})$ and therefore, $V(x,u(x))\in C^{\alpha}(\overline{\Omega})$. Furthermore,
\[
\|V(\cdot,u(\cdot))\|_{ C^{\alpha}(\overline{\Omega})}\le \sup_{|z|=2C\|u\|_{C^{\alpha}(\overline{\Omega})}}\|V(\cdot, z)\|_{C^{\alpha}(\overline{\Omega})}.
\]

We next claim that the map  $F$ in \eqref{eq_app_ref_2} is holomorphic.  To this end, let us observe that since $F$ is clearly locally bounded, it suffices verify the weak holomorphy,  see \cite[p. 133]{Poschel_Trub_book}. In doing so, let $(f_0, u_0), (f,u)\in B_1\times B_2$, and let us show that the map
\[
 \lambda\mapsto F((f_0, u_0)+\lambda (f,u))
\]
is holomorphic in $\C$ with values in $B_3$.  Clearly, we only have to check that the map $\lambda\mapsto V(x,u_0(x)+\lambda u_1(x))$ is holomorphic in $\C$ with values in $C^\alpha (\overline{\Omega})$. This follows from the fact that the series
\[
\sum_{k=1}^\infty  \frac{V_k}{k!} (u_0+\lambda u_1)^k
\]
converges in $C^\alpha(\overline{\Omega})$, locally  uniformly in  $\lambda\in \C$, see \eqref{eq_app_ref_4}.

We have $F(0,0)=0$ and the partial differential $\p_u F(0,0): B_2\to B_3$ is given by
\[
\p_u F(0,0)v=(-\Delta v+V_1 v , v|_{\p \Omega}).
\]
In view of (c), an application of \cite[Theorem 6.15]{Gil_Tru_book} allows us to conclude that the map  $\p_u F(0,0): B_2\to B_3$ is a linear isomorphism.

By the implicit function theorem, see \cite[p. 144]{Poschel_Trub_book}, we get that there exists  $\delta>0$ and a unique holomorphic map $S: B_\delta(\p \Omega)\to C^{2,\alpha}(\overline{\Omega})$ such that $S(0)=0$ and $F(f, S(f))=0$ for all $f\in B_\delta(\p \Omega)$. Setting $u=S(f)$ and noting that $S$ is Lipschitz continuous and $S(0)=0$, we see that
\[
\|u\|_{C^{2,\alpha}(\overline{\Omega})}\le C\|f\|_{C^{2,\alpha}(\overline{\Omega})}.
\]
The proof is complete.
\end{proof}

\begin{cor}
The map
\[
B_\delta(\p \Omega)\to C^{1,\alpha}(\overline{\Omega}), \quad f\mapsto \p_\nu u_f|_{\p \Omega}
\]
is holomorphic.
\end{cor}

\end{appendix}

\section*{Acknowledgements}
K.K. is very grateful to Mikko Salo for providing the references \cite{Browder}, \cite{LLS_poisson},  and \cite{LLS_conformal}, and to Daniel Tataru for a very helpful discussion on semilinear PDE. The research of K.K. is partially supported by the National Science Foundation (DMS 1815922). The research of G.U. is partially supported by NSF and a Si-Yuan Professorship of HKUST. Part of the work was supported by the NSF grant DMS-1440140 while both authors  were in residence at MSRI in Berkeley, California, during Fall 2019 semester.

\end{document}